\documentclass{article}

\usepackage{amsthm,amsfonts,amssymb,amsmath}

\usepackage{a4wide}
\usepackage{cite}
\usepackage{bm}

\newtheorem{theorem}{Theorem}

\newtheorem{definition}[theorem]{Definition}
\newtheorem{example}[theorem]{Example}

\newtheorem{remark}[theorem]{Remark}

\newcommand{\Hl}{{_{a}}\textsl{J}_t^\alpha}
\newcommand{\Hr}{{_{t}}\textsl{J}_b^\alpha}

\begin{document}

\title{A Generalized Fractional Calculus of Variations\thanks{This is a preprint 
of a paper whose final and definitive form will appear
in \emph{Control and Cybernetics}. Paper submitted 01-Oct-2012;
revised 25-March-2013; accepted for publication 17-April-2013.}}

\author{Tatiana Odzijewicz$^{1}$\\
\texttt{tatianao@ua.pt}
\and
Agnieszka B. Malinowska$^{2}$\\
\texttt{a.malinowska@pb.edu.pl}
\and
Delfim F. M. Torres$^{1}$\\
\texttt{delfim@ua.pt}}

\date{$^1$Center for Research and Development in Mathematics and Applications\\
Department of Mathematics, University of Aveiro, 3810-193 Aveiro, Portugal\\[0.3cm]
$^2$Faculty of Computer Science, Bialystok University of Technology\\
15-351 Bia\l ystok, Poland}

\maketitle


\begin{abstract}
We study incommensurate fractional variational problems
in terms of a generalized fractional integral with
Lagrangians depending on classical derivatives and
generalized fractional integrals and derivatives.
We obtain necessary optimality conditions for the basic
and isoperimetric problems, transversality
conditions for free boundary value problems,
and a generalized Noether type theorem.

\bigskip


\bigskip


\end{abstract}


\section{Introduction}

Till recently, it was believed that Lagrangian and Hamiltonian mechanics
were not valid in the presence of nonconservative forces such as friction \cite{book:Lanezos}.
In the last years, however, several approaches have been investigated in order to find a Lagrangian
or a Hamiltonian description for classes of dissipative (or dissipative-looking) systems
\cite{MR2609969,MR2642398,K:R:S,MR2433429,MR2531292}.
One possibility to have a Lagrangian and a Hamiltonian
formulation, for both conservative and nonconservative systems,
was proposed by Fred Riewe in 1996 and consists
in using fractional derivatives \cite{CD:Riewe:1996,CD:Riewe:1997}.
Riewe's papers \cite{CD:Riewe:1996,CD:Riewe:1997} gave rise to a new and important
research field, called \emph{the fractional calculus of variations} \cite{book:frac}.
Nowadays the subject is of strong interest, and many results of variational analysis
were extended to the non-integer case (see, e.g.,
\cite{MyID:182,MyID:209,DerInt,MyID:152,MyID:179,Cresson,fred:tor,MyID:181,MyID:207,MyID:203}).
Here we study problems of calculus of variations with generalized fractional operators
\cite{OmPrakashAgrawal,FVC_Gen_Int,FVC_Gen}. Generalized fractional integrals are given as
a linear combination of left and right fractional integrals with general kernels.
Generalized fractional Riemann--Liouville and Caputo derivatives are defined as a composition
of classical derivatives and generalized fractional integrals. In a first problem,
we ask how to determine the extremizers of a functional defined by a generalized fractional integral
involving $n$ generalized fractional Caputo derivatives and $n$ generalized fractional integrals.
All these operators have different (non-integer) orders. We obtain necessary optimality conditions,
and in the case of free boundary values, also natural boundary conditions. Next, we derive
Euler--Lagrange type equations for an extended isoperimetric problem
and we obtain a Noether type theorem.

The text is organized as follows. In Section~\ref{sec:prelim} we give the definitions
and main properties of the generalized fractional operators. We prove Euler--Lagrange equations
for the fundamental generalized problem in Section~\ref{sec:fp}, and natural boundary conditions
for free boundary value problems in Section~\ref{sec:fpfb}. Section~\ref{sec:fp:iso}
is devoted to the generalized isoperimetric problem and Section~\ref{sec:NT} to Noether's theorem.
Finally, in Section~\ref{sec:appl:phys} we present an application of our results
to the damped harmonic oscillator.


\section{Preliminaries}
\label{sec:prelim}

We start by defining the generalized fractional operators \cite{OmPrakashAgrawal}.
As particular cases, by choosing appropriate kernels, such operators are reduced
to the standard fractional integrals and derivatives of fractional calculus
(see, e.g., \cite{book:Kilbas,book:Klimek,book:Podlubny}).
Throughout the text, $\alpha$ denotes a real number between zero and one.
Following \cite{MyID:209}, we use round brackets for the arguments of functions,
and square brackets for the arguments of operators.

\begin{definition}[The generalized fractional integral]
The operator $K_P^\alpha$ is given by
\begin{equation*}
K_P^{\alpha}\left[f\right](x)
:= K_P^{\alpha}\left[t \mapsto f(t)\right](x)
=p\int\limits_{a}^{x}k_{\alpha}(x,t)f(t)dt
+q\int\limits_{x}^{b}k_{\alpha}(t,x)f(t)dt,
\end{equation*}
where $P=\langle a,x,b,p,q\rangle$ is the \emph{parameter set} ($p$-set for brevity),
$x\in[a,b]$, $p,q$ are real numbers, and $k_{\alpha}(x,t)$
is a kernel which may depend on $\alpha$.
The operator $K_P^\alpha$ is referred as the \emph{operator $K$} ($K$-op for simplicity)
of order $\alpha$ and $p$-set $P$.
\end{definition}

Note that if we define
\[
G(x,t):= \left\{ \begin{array}{ll}
p k_\alpha(x,t) & \mbox{if $t < x$},\\
q k_\alpha(t,x) & \mbox{if $t \geq x$},
\end{array} \right.
\]
then the operator $K_P^\alpha$ can be written in the form
\begin{equation*}
K_P^{\alpha}\left[f\right](x)
= K_P^{\alpha}\left[ t \mapsto f(t)\right](x)
=\int_a^b G(x,t) f(t) dt.
\end{equation*}
Thus, the generalized fractional integral is a Fredholm operator,
one of the oldest and most respectable class of operators
that arise in the theory of integral equations
\cite{book:Helemskii,book:Polyanin}.

\begin{example}
\begin{enumerate}
\item Let $k_\alpha(t-\tau)
=\frac{1}{\Gamma(\alpha)}(t-\tau)^{\alpha-1}$ and $0<\alpha<1$.
If $P=\langle a,t,b,1,0\rangle$, then
\begin{equation*}
K_{P}^\alpha[f](t)=\frac{1}{\Gamma(\alpha)}
\int\limits_a^t(t-\tau)^{\alpha-1}f(\tau)d\tau
=: {_{a}}\textsl{I}^{\alpha}_{t}[f](t)
\end{equation*}
is the left Riemann--Liouville fractional integral
of order $\alpha$; if $P=\langle a,t,b,1,0\rangle$, then
\begin{equation*}
K_{P}^\alpha[f](t)=\frac{1}{\Gamma(\alpha)}
\int\limits_t^b(\tau-t)^{\alpha-1}f(\tau)d\tau
=: {_{t}}\textsl{I}^{\alpha}_{b} [f](t)
\end{equation*}
is the right Riemann--Liouville fractional integral
of order $\alpha$.

\item For $k_\alpha(t-\tau)
=\frac{1}{\Gamma(\alpha(t,\tau))}(t-\tau)^{\alpha(t,\tau)-1}$
and $P=\langle a,t,b,1,0\rangle$
\begin{equation*}
K_{P}^\alpha[f](t)
= \int\limits_a^t\frac{1}{\Gamma(\alpha(t,
\tau)}(t-\tau)^{\alpha(t,\tau)-1}f(\tau)d\tau
=: {_{a}}\textsl{I}^{\alpha(t,\cdot)}_{t} [f](t)
\end{equation*}
is the left Riemann--Liouville fractional integral
of variable order $\alpha(t,\tau)$,
and for $P=\langle a,t,b,0,1\rangle$
\begin{equation*}
K_{P}^\alpha[f](t)
= \int\limits_t^b\frac{1}{\Gamma(\alpha(\tau,t))}(\tau
-t)^{\alpha(t,\tau)-1}f(\tau)d\tau
=: {_{t}}\textsl{I}^{\alpha(\cdot,t)}_{b} [f](t)
\end{equation*}
is the right Riemann--Liouville fractional integral
of variable order $\alpha(t,\tau)$ \cite{IDOTA2011}.

\item For $0<\alpha<1$, $k_\alpha(t,\tau)
=\frac{1}{\Gamma(\alpha)}\left(\log\frac{t}{\tau}\right)^{\alpha-1}\frac{1}{\tau}$
and $P=\langle a,t,b,1,0\rangle$, the operator
$K_{P}^\alpha$ reduces to the left Hadamard fractional integral \cite{MR2886720},
\begin{equation*}
K_{P}^\alpha[f](t)
= \frac{1}{\Gamma(\alpha)}\int_a^t
\left(\log\frac{t}{\tau}\right)^{\alpha-1}\frac{f(\tau)d\tau}{\tau}
=: \Hl [f](t),
\end{equation*}
and for $P=\langle a,t,b,0,1\rangle$ operator $K_{P}$
reduces to the right Hadamard fractional integral,
\begin{equation*}
K_{P}^\alpha[f](t)
= \frac{1}{\Gamma(\alpha)}\int_t^b \left(
\log\frac{\tau}{t}\right)^{\alpha-1}\frac{f(\tau)d\tau}{\tau}
=:\Hr [f](t).
\end{equation*}

\item Generalized fractional integrals can be also reduced to,
e.g., Riesz, Katugampola or Kilbas fractional operators.
Their definitions can be found in \cite{Katugampola,Kilbas,book:Kilbas}.
\end{enumerate}
\end{example}

Next results yield boundedness of the generalized fractional integral.

\begin{theorem}[\textrm{cf.} Example~6 of \cite{book:Helemskii}]
Let $\alpha\in(0,1)$ and $P=\langle a,x,b,p,q\rangle$.
If $k_\alpha$ is a square integrable function
on the square $\Delta=[a,b]\times[a,b]$, then
$K_P^{\alpha}:L_2\left([a,b]\right)\rightarrow L_2\left([a,b]\right)$
is well defined, linear, and bounded operator.
\end{theorem}

\begin{theorem}[\textrm{cf.} \cite{FVC_Gen_Int,FVC_Gen}]
\label{theorem:L1}
Let $k_\alpha\in L_1\left([0,b-a]\right)$ be a difference kernel, that is,
$k_\alpha(x,t)=k_\alpha(x-t)$.
Then, $K_P^{\alpha}:L_1\left([a,b]\right)\rightarrow L_1\left([a,b]\right)$
is a well defined bounded and linear operator.
\end{theorem}

\begin{theorem}[\textrm{cf.} Theorem~2.4 of \cite{FVC_Gen}]
\label{theorem:exist}
Let $P=\langle a,x,b,p,q\rangle$. If $k_{1-\alpha}$ is a difference kernel,
$k_{1-\alpha}\in L_1\left([0,b-a]\right)$ and $f\in AC\left([a,b]\right)$,
then $K_P^{1-\alpha}[f]$ belongs to $AC\left([a,b]\right)$.
\end{theorem}

The generalized fractional derivatives $A_P^\alpha$
and $B_P^\alpha$ are defined in terms of the
generalized fractional integral $K$-op.

\begin{definition}[Generalized Riemann--Liouville fractional derivative]
\label{def:GRL}
Let $P$ be a given parameter set and $0<\alpha < 1$.
The operator $A_P^\alpha$ is defined by
$A_P^\alpha := D \circ K_P^{1-\alpha}$,
where $D$ denotes the standard derivative operator,
and is referred as the \emph{operator $A$} ($A$-op)
of order $\alpha$ and $p$-set $P$.
\end{definition}

\begin{remark}
Operator $A$ is well-defined for all functions $f$ such that
$K_P^{1-\alpha}[f]$ is differentiable. Theorem~\ref{theorem:exist}
assure us that the domain of $A$ is nonempty.
\end{remark}

\begin{definition}[Generalized Caputo fractional derivative]
\label{def:GC}
Let $P$ be a given parameter set and $\alpha \in (0,1)$.
The operator $B_P^\alpha$ is defined by
$B_P^\alpha := K_P^{1-\alpha} \circ D$,
where $D$ denotes the standard derivative operator,
and is referred as the \emph{operator $B$} ($B$-op)
of order $\alpha$ and $p$-set $P$.
\end{definition}

\begin{remark}
Operator $B$ is well-defined for differentiable functions.
\end{remark}

\begin{example}
The standard Riemann--Liouville and Caputo fractional derivatives
(see, \textrm{e.g.}, \cite{book:Kilbas,book:Podlubny,book:Klimek})
are easily obtained from the general kernel operators $A_P^\alpha$
and $B_P^\alpha$, respectively. Let $k_\alpha(t-\tau)
=\frac{1}{\Gamma(1-\alpha)}(t-\tau)^{-\alpha}$,
$\alpha \in (0,1)$. If $P=\langle a,t,b,1,0\rangle$, then
\begin{equation*}
A_{P}^\alpha[f](t)=\frac{1}{\Gamma(1-\alpha)}
\frac{d}{dt} \int\limits_a^t(t-\tau)^{-\alpha}f(\tau)d\tau
=: {_{a}}\textsl{D}^{\alpha}_{t}[f](t)
\end{equation*}
is the standard left Riemann--Liouville fractional derivative
of order $\alpha$, while
\begin{equation*}
B_{P}^\alpha[f](t)=\frac{1}{\Gamma(1-\alpha)}
\int\limits_a^t(t-\tau)^{-\alpha} f'(\tau)d\tau
=: {^{C}_{a}}\textsl{D}^{\alpha}_{t}[f](t)
\end{equation*}
is the standard left Caputo fractional derivative of order $\alpha$;
if $P=\langle a,t,b,0,1\rangle$, then
\begin{equation*}
- A_{P}^\alpha[f](t)
=- \frac{1}{\Gamma(1-\alpha)} \frac{d}{dt}
\int\limits_t^b(\tau-t)^{-\alpha}f(\tau)d\tau
=: {_{t}}\textsl{D}^{\alpha}_{b}[f](t)
\end{equation*}
is the standard right Riemann--Liouville
fractional derivative of order $\alpha$, while
\begin{equation*}
- B_{P}^\alpha[f](t) = - \frac{1}{\Gamma(1-\alpha)}
\int\limits_t^b(\tau-t)^{-\alpha} f'(\tau)d\tau
=: {^{C}_{t}}\textsl{D}^{\alpha}_{b} [f](t)
\end{equation*}
is the standard right Caputo fractional derivative of order $\alpha$.
\end{example}

The following theorems give integration by parts formulas for operators
$A$, $B$ and $K$. For detailed proofs we refer
the reader to \cite{FVC_Gen_Int,FVC_Gen}.

\begin{theorem}
\label{thm:gfip:Kop}
Let $\alpha \in (0,1)$, $P=\langle a,t,b,p,q\rangle$,
$k_{\alpha}$ be a square-integrable function
on $\Delta=[a,b]\times[a,b]$, and $f,g\in L_2\left([a,b]\right)$.
The generalized fractional integral $K_P^{\alpha}$ satisfies
the integration by parts formula
\begin{equation}
\label{eq:fracIP:K}
\int\limits_a^b g(x)K_P^{\alpha}\left[f\right](x)dx
=\int\limits_a^b f(x)K_{P^*}^{\alpha}\left[g\right](x)dx,
\end{equation}
where $P^{*}=<a,t,b,q,p>$.
\end{theorem}

\begin{theorem}
\label{thm:gfip}
Let $\alpha \in (0,1)$, $P=\langle a,t,b,p,q\rangle$, and $k_{\alpha}$
be a square integrable function on $\Delta=[a,b]\times [a,b]$.
If functions $f,K_{P^*}^{1-\alpha}[g]  \in AC([a,b])$, then
\begin{equation}
\label{eq:fip:2}
\int\limits_a^b g(x) B_{P}^\alpha \left[f\right](x)dx
=\left. f(x) K_{P^*}^{1-\alpha}\left[g\right](x)\right|_a^b
-\int_a^b f(x) A_{P^*}^\alpha\left[g\right](x)dx,
\end{equation}
where $P^*=<a,t,b,q,p>$.
\end{theorem}

\begin{theorem}
\label{thm:IPL1}
Let $0<\alpha<1$, $P=\langle a,x,b,p,q\rangle$, and $k_\alpha$
be a difference kernel such that $k_{\alpha}\in L_1[0,b-a]$.
If $f\in L_1\left([a,b]\right)$ and $g\in C\left([a,b]\right)$,
then the operator $K_P^{\alpha}$ satisfies the integration
by parts formula \eqref{eq:fracIP:K}.
\end{theorem}

\begin{theorem}
\label{thm:gfip2}
Let $\alpha \in (0,1)$, $P=\langle a,t,b,p,q\rangle$, and $k_{\alpha}\in L_1\left([0,b-a]\right)$
be a difference kernel. If functions $f, g \in AC([a,b])$, then formula \eqref{eq:fip:2} holds.
\end{theorem}

For $\mathbf{f}=\left[f_1,\dots,f_N\right]:[a,b]\rightarrow\mathbb{R}^N$,
where $N\in\mathbb{N}$, we put
\begin{equation*}
\begin{split}
A_{P}^\alpha \left[\mathbf{f}\right](x)
&:=\left[A_{P}^\alpha \left[f_1\right](x),
\dots,A_{P}^\alpha \left[f_N\right](x)\right],\\
B_{P}^\alpha \left[\mathbf{f}\right](x)
&:=\left[B_{P}^\alpha \left[f_1\right](x),
\dots,B_{P}^\alpha \left[f_N\right](x)\right],\\
K_{P}^\alpha \left[\mathbf{f}\right](x)
&:=\left[K_{P}^\alpha \left[f_1\right](x),
\dots,K_{P}^\alpha \left[f_N\right](x)\right].
\end{split}
\end{equation*}


\section{The generalized fundamental variational problem}
\label{sec:fp}

We consider the problem of finding a function
$\mathbf{y}=\left[y_1,\dots,y_N\right]$
that gives an extremum (minimum or maximum)
to the functional
\begin{equation}
\label{eq:1}
\mathcal{J}(\mathbf{y})=K_{P}^\alpha\left[t \mapsto
F\left(t,\mathbf{y}(t),\mathbf{y}'(t),B_{P_1}^{\beta_1}
\left[\mathbf{y}\right](t),\dots,B_{P_n}^{\beta_n} \left[\mathbf{y}\right](t),
K_{R_1}^{\gamma_1}\left[\mathbf{y}\right](t),
\dots, K_{R_m}^{\gamma_m}\left[\mathbf{y}\right](t)\right)\right](b)
\end{equation}
subject to the boundary conditions
\begin{equation}
\label{eq:2}
\mathbf{y}(a)=\mathbf{y}_a, \quad \mathbf{y}(b)=\mathbf{y}_b,
\end{equation}
where $\alpha,\beta_i, \gamma_k\in(0,1)$, $P=<a,b,b,1,0>$,
$P_i=<a,t,b,p_i,q_i>$, and $R_k=<a,t,b,r_k,s_k>$, $i=1,\dots,n$, $k=1,\dots,m$.
For simplicity of notation, we introduce the operator
$\left\{ \cdot \right\}_{P_D, R_I}^{\beta,\gamma}$ defined by
\begin{equation*}
\left\{\mathbf{y}\right\}_{P_D, R_I}^{\beta,\gamma}(t)
:=\left(t,\mathbf{y}(t),\mathbf{y}'(t),B_{P_D}^\beta
\left[\tau \mapsto \mathbf{y}(\tau)\right](t),
K_{R_I}^\gamma \left[\tau \mapsto \mathbf{y}(\tau)\right](t)\right),
\end{equation*}
where
\begin{equation*}
B_{P_D}^\beta:=\left(B_{P_1}^{\beta_1}, \dots,B_{P_n}^{\beta_n}\right),
\quad K_{R_I}^\gamma:=\left(K_{R_1}^{\gamma_1},\dots, K_{R_m}^{\gamma_m}\right).
\end{equation*}
The operator $K_{P}^\alpha$ has kernel $k_\alpha(x,t)$
and, for $i=1,\dots,n$ and $k=1,\dots,m$, operators $B_{P_i}^{\beta_i}$
and $K_{R_k}^{\gamma_k}$ have kernels $h_{1-\beta_i}(t,\tau)$
and $h_{\gamma_k}(t,\tau)$, respectively.
In the sequel we assume that:
\begin{enumerate}
\item[(H1)] the Lagrangian $F\in C^1\left([a,b]
\times\mathbb{R}^{N\times(n+m+2)};\mathbb{R}\right)$;

\item[(H2)] functions
$D\left[t \mapsto \partial_{N+j}F\left\{\mathbf{y}\right\}_{P_D, R_I}^{\beta,\gamma}(t)
k_\alpha(b,t)\right]$,
$A_{P_i^*}^{\beta_i}\left[\tau
\mapsto k_\alpha(b,\tau)\partial_{(i+1)N+j}
F\left\{\mathbf{y}\right\}_{P_D, R_I}^{\beta,\gamma}(\tau)\right]$,
$K_{R_k^*}^{\gamma_k}\left[\tau \mapsto k_\alpha(b,\tau)
\partial_{(n+1+k)N+j} F\left\{\mathbf{y}\right\}_{P_D, R_I}^{\beta,\gamma}(\tau)\right]$
and $t \mapsto k_\alpha(b,t)\partial_{j}
F \left\{\mathbf{y}\right\}_{P_D, R_I}^{\beta,\gamma}(t)$
are continuous on $(a,b)$, $j=2,\dots,N+1$, $i=1,\dots,n$, $k=1,\dots,m$;

\item[(H3)] functions $t \mapsto \partial_{N+j}
F\left\{\mathbf{y}\right\}_{P_D, R_I}^{\beta,\gamma}(t)k_\alpha(b,t)$ and
$K_{P_i^*}^{1-\beta_i}\left[\tau \mapsto k_\alpha(b,\tau)
\partial_{(i+1)N+j} F\left\{\mathbf{y}\right\}_{P_D, R_I}^{\beta,\gamma}(\tau)\right]$
$\in AC([a,b])$, $j=2,\dots,N+1$, $i=1,\dots,n$;

\item[(H4)] for $i=1,\dots,n$, $k=1,\dots,m$, the kernels $k_\alpha(x,t)$,
$h_{1-\beta_i}(t,\tau)$ and $h_{\gamma_k}(t,\tau)$
are such that we are able to use Theorems~\ref{thm:gfip:Kop},
\ref{thm:gfip}, \ref{thm:IPL1} and/or \ref{thm:gfip2}.
\end{enumerate}

\begin{definition}
A function $ \mathbf{y}\in C^1\left([a,b];\mathbb{R}^N\right)$ is said to be
admissible for the fractional variational problem \eqref{eq:1}--\eqref{eq:2}
if functions $B_{P_i}^{\beta_i}[\mathbf{y}]$ and $K_{R_k}^{\gamma_k}[\mathbf{y}]$,
$i=1,\dots,n$, $k=1,\dots,m$ exist and are continuous on the interval $[a,b]$,
and $\mathbf{y}$ satisfies the given boundary conditions \eqref{eq:2}.
\end{definition}

\begin{theorem}
\label{theorem:ELCaputo}
If $\mathbf{y}$ is a solution to problem \eqref{eq:1}--\eqref{eq:2},
then $\mathbf{y}$ satisfies the system of generalized Euler--Lagrange equations
\begin{multline}
\label{eq:eqELCaputo}
k_\alpha(b,t)\partial_j F \left\{\mathbf{y}\right\}_{P_D, R_I}^{\beta,\gamma}(t)
-\sum_{i=1}^n A_{P_i^*}^{\beta_i}\left[\tau \mapsto k_\alpha(b,\tau)\partial_{(i+1)N+j}
F\left\{\mathbf{y}\right\}_{P_D, R_I}^{\beta,\gamma}(\tau)\right](t)\\
+\sum_{k=1}^m K_{R_k^*}^{\gamma_k}\left[\tau \mapsto k_\alpha(b,\tau)\partial_{(n+1+k)N+j}
F\left\{\mathbf{y}\right\}_{P_D, R_I}^{\beta,\gamma}(\tau)\right](t)
-\frac{d}{dt}\left(\partial_{N+j}F\left\{\mathbf{y}\right\}_{P_D, R_I}^{\beta,\gamma}(t)
k_\alpha(b,t)\right) =0
\end{multline}
for all $t\in(a,b)$, $j=2,\dots,N+1$.
\end{theorem}

\begin{proof}
The proof is analogous to that of \cite[Theorem~4.2]{FVC_Gen_Int}.
\end{proof}


\section{Generalized free-boundary variational problem}
\label{sec:fpfb}

Assume now that in problem \eqref{eq:1}--\eqref{eq:2}
the boundary conditions \eqref{eq:2} are substituted by
\begin{equation}
\label{eq:Free1}
\mathbf{y}(a) \textnormal{ is free }
\textnormal{ and } \mathbf{y}(b)=\mathbf{y}_b.
\end{equation}

\begin{theorem}
\label{theorem:NatBound}
If $\mathbf{y}$ is a solution to the problem of extremizing functional \eqref{eq:1}
with \eqref{eq:Free1} as the boundary conditions, then $\mathbf{y}$ satisfies
the system of Euler--Lagrange equations \eqref{eq:eqELCaputo}. Moreover,
the extra system of natural boundary conditions
\begin{equation}
\label{eq:NatBoundCond}
\partial_{N+j} F\left\{\mathbf{y}\right\}_{P_D,R_I}^{\beta,\gamma}(a)k_\alpha(b,a)
+\sum_{i=1}^nK_{P_i^*}^{1-\beta_i}\left[\tau \mapsto \partial_{(i+1)N+j}
F\left\{\mathbf{y}\right\}_{P_D, R_I}^{\beta,\gamma}(\tau)k_\alpha(b,\tau)\right](a)=0,
\end{equation}
$j=2,\dots,N+1$, holds.
\end{theorem}

\begin{proof}
The proof is analogous to that of \cite[Theorem~5.1]{FVC_Gen_Int}.
\end{proof}


\section{Generalized isoperimetric problem}
\label{sec:fp:iso}

Let $\xi\in\mathbb{R}$. Among all functions
$\mathbf{y}:[a,b]\rightarrow\mathbb{R}^N$
satisfying the boundary conditions
\begin{equation}
\label{eq:IsoBound}
\mathbf{y}(a)=\mathbf{y}_a, \quad \mathbf{y}(b)=\mathbf{y}_b,
\end{equation}
and an isoperimetric constraint of the form
\begin{equation}
\label{eq:IsoConstr}
\mathcal{I}\left(\mathbf{y}\right)
=K_{P}^\alpha\left[G\left\{\mathbf{y}\right\}_{P_D, R_I}^{\beta,\gamma}\right](b)=\xi,
\end{equation}
we look for those that extremize (\textrm{i.e.}, minimize or maximize) the functional
\begin{equation}
\label{eq:IsoFunct}
\mathcal{J}\left(\mathbf{y}\right)
=K_{P}^\alpha\left[F\left\{\mathbf{y}\right\}_{P_D, R_I}^{\beta,\gamma}\right](b).
\end{equation}
For $i=1,\dots, n$, $k=1,\dots,m$ operators $K_{P}^\alpha$, $B_{P_i}^{\beta_i}$ and $K_{R_k}^{\gamma_k}$,
as well as function $F$, are the same as in problem \eqref{eq:1}--\eqref{eq:2}.
Moreover, we assume that functional \eqref{eq:IsoConstr} satisfies hypotheses (H1)--(H4).

\begin{definition}
A function $\mathbf{y} : [a,b]\to\mathbb R^N$
is said to be \emph{admissible} for problem \eqref{eq:IsoBound}--\eqref{eq:IsoFunct}
if functions $B_{P_i}^{\beta_i}[\mathbf{y}]$ and $K_{R_k}^{\gamma_k}[\mathbf{y}]$,
$i=1,\dots,n$, $k=1,\dots,m$, exist and are continuous on $[a,b]$, and $\mathbf{y}$ satisfies
the given boundary conditions \eqref{eq:IsoBound} and the isoperimetric constraint \eqref{eq:IsoConstr}.
\end{definition}

\begin{definition}
An admissible function $\mathbf{y}\in C^1\left([a,b],\mathbb{R}^N\right)$
is said to be an \emph{extremal} for $\mathcal{I}$ if it satisfies
the system of Euler--Lagrange equations \eqref{eq:eqELCaputo}
associated with functional in \eqref{eq:IsoConstr}.
\end{definition}

\begin{theorem}
\label{theorem:EL2}
If $\mathbf{y}$ is a solution to the isoperimetric problem
\eqref{eq:IsoBound}--\eqref{eq:IsoFunct} and
is not an extremal for $\mathcal{I}$, then
there exists a real constant $\lambda$ such that
\begin{multline*}
k_\alpha(b,t)\partial_j H \left\{\mathbf{y}\right\}_{P_D,R_I}^{\beta,\gamma}(t)
+\sum_{k=1}^m K_{R_k^*}^{\gamma_k}\left[\tau \mapsto k_\alpha(b,\tau)\partial_{(n+1+k)N+j}
H\left\{\mathbf{y}\right\}_{P_D,R_I}^{\beta,\gamma}(\tau)\right](t)\\
-\sum_{i=1}^n A_{P_i^*}^{\beta_i}\left[\tau \mapsto k_\alpha(b,\tau)\partial_{(i+1)N+j}
H\left\{\mathbf{y}\right\}_{P_D,R_I}^{\beta,\gamma}(\tau)\right](t)
-\frac{d}{dt}\left(\partial_{j+N}
H\left\{\mathbf{y}\right\}_{P_D,R_I}^{\beta,\gamma}(t)k_\alpha(b,t)\right)=0
\end{multline*}
for all $t\in(a,b)$, $j=2,\dots,N+1$,
where $H(t,y,u,v,w)=F(t,y,u,v,w)-\lambda G(t,y,u,v,w)$,
$P_i^*=<a,t,b,q_i,p_i>$, and $R_k^*=<a,t,b,s_k,r_k>$.
\end{theorem}

\begin{proof}
The proof is analogous to that of \cite[Theorem~6.3]{FVC_Gen_Int}.
\end{proof}


\section{Generalized fractional Noether's theorem}
\label{sec:NT}

Emmy Noether's theorem on extremal functionals,
establishing that certain symmetries imply
conservation laws (constants of motion),
has been called ``the most important theorem
in physics since the Pythagorean theorem''. For a recent account
of Noether's theorem and possible applications in physics, from many different
points of view, we refer the reader to \cite{1239.00018}.
Formulations in the more general context of optimal control can be found
in \cite{GouveiaTorresRochaPolonia05,del:EJC}.
Conservation laws appear naturally in closed systems.
In presence of non-conservative or dissipative forces,
the constants of motion are broken and
Noether's classical theorem ceases to be valid.
It is still possible, however, to obtain Noether type
theorems that cover both conservative and
non-conservative cases. Roughly speaking,
one can prove that Noether's conservation laws are
still valid if a new term, involving the non-conservative forces, is
added to the standard constants of motion \cite{MyID:062}.
The first Noether theorem for the fractional calculus of variations
was obtained in 2007 \cite{MyID:068}. Since then,
the subject attracted a lot of attention. The state of the art
is given in the book \cite{book:frac}. Here we obtain a
Noether's theorem for generalized fractional variational problems.

\begin{definition}
We say that the functional \eqref{eq:1} is invariant under an
$\varepsilon$-parameter group of infinitesimal transformations
\begin{equation}
\label{trans:1}
\hat{\mathbf{y}}(t)=\mathbf{y}(t)+\varepsilon {\bm \xi}(t,\mathbf{y}(t))+o(\varepsilon)
\end{equation}
if for any subinterval $[t_a,t_b]\subseteq [a,b]$ one has
\begin{equation}
\label{inv}
K_{\bar{P}}^\alpha \left[t \mapsto F\left\{\mathbf{y}\right\}_{P_D,R_I}^{\beta,\gamma}(t)\right](t_b)
=K_{\bar{P}}^\alpha\left[t \mapsto F\left\{\hat{\mathbf{y}}\right\}_{P_D,R_I}^{\beta,\gamma}(t)\right](t_b),
\end{equation}
where $\bar{P}=\langle t_a,t_b,t_b,1,0 \rangle$.
\end{definition}

\begin{theorem}
If functional \eqref{eq:1} is invariant under an $\varepsilon$-parameter
group of infinitesimal transformations, then
\begin{multline}
\label{NCI}
\sum\limits_{j=2}^{N+1}\Biggl(\partial_j
F\left\{\mathbf{y}\right\}_{P_D,R_I}^{\beta,\gamma}(t)\cdot\xi_{j-1}(t,\mathbf{y}(t))
+\partial_{N+j} F\left\{\mathbf{y}\right\}_{P_D,R_I}^{\beta,\gamma}(t)
\cdot\frac{d}{dt}\xi_{j-1}(t,\mathbf{y}(t))\\
+\sum\limits_{i=1}^n\partial_{(i+1)N+j}
F\left\{\mathbf{y}\right\}_{P_D,R_I}^{\beta,\gamma}(t)
\cdot B_{P_i}^{\beta_i}[\tau \mapsto \xi_{j-1}(\tau,\mathbf{y(\tau)})](t)\\
+\sum\limits_{k=1}^m\partial_{(n+1+k)N+j}
F\left\{\mathbf{y}\right\}_{P_D,R_I}^{\beta,\gamma}(t)
\cdot K_{R_i}^{\gamma_i}[\tau \mapsto \xi_{j-1}(\tau,\mathbf{y}(\tau))](t)\Biggr) = 0.
\end{multline}
\end{theorem}

\begin{proof}
Since, by hypothesis, condition \eqref{inv}
is satisfied for any subinterval $[t_a,t_b]\subseteq [a,b]$, we have
\begin{equation}
\label{eq:5}
F\left\{\mathbf{y}\right\}_{P_D,R_I}^{\beta,\gamma}(t)
=F\left\{\hat{\mathbf{y}}\right\}_{P_D,R_I}^{\beta,\gamma}(t).
\end{equation}
Differentiating \eqref{eq:5} with respect to $\varepsilon$,
then putting $\varepsilon=0$, and applying definitions
and properties of generalized fractional operators, we obtain \eqref{NCI}.
\end{proof}

In order to state the Noether theorem in a compact form,
we introduce the following operators:
\begin{equation}
\label{op1}
\mathbf{D}_P^{\alpha}[f,g](t)
:= \frac{1}{k_{\alpha}(b,t)}f(t)\cdot A_{P^*}^{\alpha}[g](t)
+g(t)\cdot B_{P}^{\alpha}[f](t),
\end{equation}
\begin{equation}
\label{op2}
\mathbf{I}_P^{\alpha}[f,g](t)
:= \frac{-1}{k_{\alpha}(b,t)}f(t)
\cdot K_{P^*}^{\alpha}[g](t)+g(t)\cdot K_{P}^{\alpha}[f](t),
\end{equation}
where $P^*$ denotes the dual $p$-set of $P$, that is,
if $P=\langle a,t,b,p,q\rangle$, then  $P^*=\langle a,t,b,q,p\rangle$.

\begin{theorem}[Generalized fractional Noether's theorem]
\label{thm:gfnt}
If functional \eqref{eq:1} is invariant under an $\varepsilon$-parameter
group of infinitesimal transformations \eqref{trans:1}, then
\begin{multline}
\label{eq:NTH}
\sum\limits_{j=2}^{N+1}\Biggl(\sum\limits_{i=1}^n
\mathbf{D}_{P_i}^{\beta_i}\left[\tau \mapsto
\xi_{j-1}(\tau,\mathbf{y}(\tau)), \tau \mapsto
k_{\alpha}(b,\tau)\partial_{(i+1)N+j}
F\left\{\mathbf{y}\right\}_{P_D,R_I}^{\beta,\gamma}(\tau)\right](t)\\
+\sum\limits_{k=1}^m \mathbf{I}_{R_k}^{\gamma_k}\left[\tau \mapsto
\xi_{j-1}(\tau,\mathbf{y}(\tau)), \tau \mapsto
k_{\alpha}(b,\tau)\partial_{(n+1+k)N+j}
F\left\{\mathbf{y}\right\}_{P_D,R_I}^{\beta,\gamma}(\tau)\right](t)\\
+\frac{d}{dt}\left(\xi_{j-1}\left(t,\mathbf{y}(t)\right)\cdot\partial_{N+j}
F\left\{\mathbf{y}\right\}_{P_D,R_I}^{\beta,\gamma}(t)\right)
+\xi_{j-1}(t,\mathbf{y}(t))\cdot\partial_{N+j}
F\left\{\mathbf{y}\right\}_{P_D,R_I}^{\beta,\gamma}(t)
\cdot\frac{1}{k_{\alpha}(b,t)}\frac{d}{dt}k_{\alpha}(b,t)\Biggr) = 0
\end{multline}
for any generalized fractional extremal $\mathbf{y}$ of $\mathcal{J}$
and for all $t \in (a,b)$.
\end{theorem}

\begin{proof}
By Theorem~\ref{theorem:ELCaputo} we have
\begin{multline}
\label{eq:6}
k_\alpha(b,t)\partial_j F \left\{\mathbf{y}\right\}_{P_D, R_I}^{\beta,\gamma}(t)
=\sum_{i=1}^n A_{P_i^*}^{\beta_i}\left[\tau \mapsto k_\alpha(b,\tau)\partial_{(i+1)N+j}
F\left\{\mathbf{y}\right\}_{P_D, R_I}^{\beta,\gamma}(\tau)\right](t)\\
-\sum_{k=1}^m K_{R_i^*}^{\gamma_k}\left[\tau \mapsto k_\alpha(b,\tau)\partial_{(n+1+k)N+j}
F\left\{\mathbf{y}\right\}_{P_D, R_I}^{\beta,\gamma}(\tau)\right](t)
+\frac{d}{dt}\left(\partial_{N+j}F\left\{\mathbf{y}\right\}_{P_D, R_I}^{\beta,\gamma}(t)
k_\alpha(b,t)\right)
\end{multline}
for all $t\in(a,b)$, $j=2,\dots,N+1$.
Substituting \eqref{eq:6} into \eqref{inv}, we obtain
\begin{multline*}
\sum\limits_{j=2}^{N+1}\Biggl[\frac{1}{k_{\alpha}(b,t)}
\cdot\xi_{j-1} (t,\mathbf{y}(t))\Biggl(\sum_{i=1}^n A_{P_i^*}^{\beta_i}\left[
\tau \mapsto k_\alpha(b,\tau)\partial_{(i+1)N+j}
F\left\{\mathbf{y}\right\}_{P_D, R_I}^{\beta,\gamma}(\tau)\right](t)\\
-\sum_{k=1}^m K_{R_i^*}^{\gamma_k}\left[\tau \mapsto k_\alpha(b,\tau)\partial_{(n+1+k)N+j}
F\left\{\mathbf{y}\right\}_{P_D, R_I}^{\beta,\gamma}(\tau)\right](t)
+\frac{d}{dt}\left(\partial_{N+j}F\left\{\mathbf{y}\right\}_{P_D, R_I}^{\beta,\gamma}(t)
k_\alpha(b,t)\right)\Biggr)\\
+\partial_{N+j} F\left\{\mathbf{y}\right\}_{P_D,R_I}^{\beta,\gamma}(t)
\cdot\frac{d}{dt}\xi_{j-1}\left(t,\mathbf{y}(t)\right)
+\sum\limits_{i=1}^n\partial_{(i+1)N+j} F\left\{\mathbf{y}\right\}_{P_D,R_I}^{\beta,\gamma}(t)
\cdot B_{P_i}^{\beta_i}[\tau \mapsto \xi_{j-1}\left(\tau,\mathbf{y}(\tau)\right)](t)\\
+\sum\limits_{k=1}^m\partial_{(n+1+k)N+j}
F\left\{\mathbf{y}\right\}_{P_D,R_I}^{\beta,\gamma}(t)
\cdot K_{R_i}^{\gamma_i}[\tau \mapsto \xi_{j-1}(\tau,\mathbf{y}(\tau))](t)\Biggr] = 0.
\end{multline*}
Finally, we arrive to \eqref{eq:NTH} by \eqref{op1} and \eqref{op2}.
\end{proof}

\begin{example}
Let $P=\langle a,t,b,p,q\rangle$.
Consider the following problem:
\begin{equation}
\label{eq:example:2}
\begin{gathered}
\mathcal{J}[y]=\int_a^b F\left(t,B_{P}^\alpha[y](t)\right) dt \longrightarrow \min\\
y(a)=y_a \, , \quad y(b)=y_b,
\end{gathered}
\end{equation}
and transformations
\begin{equation}
\label{eq:Tr:2}
\hat{y}(t)=y(t)+\varepsilon c+o(\varepsilon),
\end{equation}
where $c$ is a constant. For any $[t_a,t_b]\subseteq [a,b]$ we have
\begin{equation*}
\int_{t_a}^{t_b}F\left(t,B_{P}^\alpha[y](t)\right) dt
=\int_{t_a}^{t_b}F\left(t,B_{P}^\alpha[\hat{y}](t)\right) dt.
\end{equation*}
Therefore, $\mathcal{J}[y]$ is invariant under \eqref{eq:Tr:2}
and Theorem~\ref{thm:gfnt} asserts that
\begin{equation}
\label{eq:ex:NTH}
A_{P^*}^\alpha[\tau \rightarrow \partial_2
F\left(\tau,B_{P}^\alpha[y](\tau)\right)](t)=0
\end{equation}
along any generalized fractional extremal $y$.
Notice that equation \eqref{eq:ex:NTH}
can be written in the form
\begin{equation*}
\frac{d}{dt}\left(K_{P^*}^\alpha[\tau \rightarrow
\partial_2 F\left(\tau,B_{P}^\alpha[y](\tau)\right)](t)\right)=0.
\end{equation*}
In analogy with the classical approach, quantity
$K_{P^*}^\alpha[\tau \rightarrow \partial_2 F\left(\tau,B_{P}^\alpha[y](\tau)\right)](t)$
is called a \emph{generalized fractional constant of motion}.
\end{example}


\section{Applications to Physics}
\label{sec:appl:phys}

If the functional \eqref{eq:1} does not depend on $B$-ops and $K$-ops,
then Theorem~\ref{theorem:ELCaputo} gives the following result:
if $\mathbf{y}$ is a solution to the problem of extremizing
\begin{equation}
\label{f:ex}
\mathcal{J}(\mathbf{y})=\int_a^b
L\left(t,\mathbf{y}(t),\mathbf{y}'(t)\right) k_\alpha(b,t) dt
\end{equation}
subject to $\mathbf{y}(a)=\mathbf{y}_a$ and $\mathbf{y}(b)=\mathbf{y}_b$,
where $\alpha\in(0,1)$, then
\begin{equation}
\label{eq:FALVA}
\partial_j L\left(t,\mathbf{y}(t),\mathbf{y}'(t)\right)-\frac{d}{dt}\partial_{N+j}
L\left(t,\mathbf{y}(t),\mathbf{y}'(t)\right)=\frac{1}{k_{\alpha}(b,t)}
\cdot \frac{d}{dt}k_{\alpha}(b,t)\partial_{N+j} L\left(t,\mathbf{y}(t),\mathbf{y}'(t)\right),
\end{equation}
$j=2,\dots,N+1$. In addition, if we assume that functional \eqref{f:ex}
is invariant under transformations \eqref{trans:1}, then Noether's theorem yields that
\begin{equation*}
\sum\limits_{j=2}^{N+1}\biggl(\frac{d}{dt}\left(\xi_{j-1}(t,\mathbf{y}(t))\cdot\partial_{N+j}
L(t,\mathbf{y}(t),\mathbf{y}'(t))\right)
+\xi_{j-1}(t,\mathbf{y}(t))\cdot\partial_{N+j}
L(t,\mathbf{y}(t),\mathbf{y}'(t))
\cdot\frac{1}{k_{\alpha}(b,t)}\frac{d}{dt}k_{\alpha}(b,t)\biggr)=0,
\end{equation*}
along any extremal of \eqref{f:ex}.
Let us consider kernel $k_\alpha(b,t)=\mathrm{e}^{\alpha (b-t)}$
and the Lagrangian for a three dimensional system:
\begin{equation*}
L\left(\mathbf{y},\dot{\mathbf{y}}\right)
=\frac{1}{2}m \left(\dot{y_1}^2+\dot{y_2}^2+\dot{y_3}^2\right)-V(\mathbf{y}),
\end{equation*}
where $V(\mathbf{y})$ is the potential energy and $m$ stands for the mass.
Observe that an explicitly time dependent integrand $\tilde{L}=\mathrm{e}^{\alpha (b-t)}L$
of functional \eqref{f:ex} is known in the literature as the Bateman--Caldirola--Kanai (BCK)
Lagrangian of a quantum dissipative system \cite{Menon,MR2476015}. But in our case the Lagrangian
of the system is $L$ and not $\mathrm{e}^{\alpha (b-t)}L$.
The Euler--Lagrange equations \eqref{eq:FALVA} give the following system of
second order ordinary differential equations:
\begin{equation*}
\begin{cases}
\ddot{y_1}(t)-\alpha \dot{y_1}(t) = -\frac{1}{m}\partial_1 V(\mathbf{y}(t))\\
\ddot{y_2}(t)-\alpha \dot{y_2}(t) = -\frac{1}{m}\partial_2 V(\mathbf{y}(t))\\
\ddot{y_3}(t)-\alpha \dot{y_3}(t) = -\frac{1}{m}\partial_3 V(\mathbf{y}(t)).
\end{cases}
\end{equation*}
If $\gamma := -\alpha$, then
\begin{equation}
\label{eq:damped}
\ddot{y}_i+\gamma\dot{y}_i+\frac{1}{m}\frac{\partial V}{\partial y_i}=0,
\end{equation}
$i=1,2,3$, which are equations for the damped motion of a three-dimensional particle
under the action of a force $\left[-\frac{\partial V}{\partial y_1},
-\frac{\partial V}{\partial y_2},-\frac{\partial V}{\partial y_3}\right]$
(see, e.g., \cite{Herrera}). Choosing $V := k\frac{y_1^2+y_2^2+y_3^2}{2}$,
we can transform \eqref{eq:damped} into equations
for a damped simple harmonic oscillator:
\begin{equation*}
\ddot{y_i}(t)+\gamma\dot{y_i}(t)+\omega^2 y_i(t)=0,
\end{equation*}
$i=1,2,3$, with $\omega^2=\frac{k}{m}$. Now, let us consider the following Lagrangian:
\begin{equation}
\label{lag:inv}
L\left(\mathbf{y},\dot{\mathbf{y}}\right)
=\frac{1}{2}m \left(\dot{y_1}^2+\dot{y_2}^2+\dot{y_3}^2\right)-mgy_3^2.
\end{equation}
We see at once that the Lagrangian \eqref{lag:inv} is invariant under the transformation
\begin{equation*}
\hat{y}_1=y_1+\varepsilon,~\hat{y}_2=y_2,~\hat{y}_3=y_3.
\end{equation*}
In this case Noether's theorem gives
\begin{equation}
\label{eq:cm}
\frac{d}{dt}(m\dot{y}_1 )= \alpha m\dot{y}_1.
\end{equation}
If $\alpha=0$, then there is no friction and \eqref{eq:cm} yields the classical conservation
of linear momentum $p_1=m\dot{y}_1=const$. Observe that the generalized momentum conjugate
to $y_i$ is $p_i=\frac{\partial L}{\partial \dot{y}_i}=m\dot{y}_i$, $i=1,2,3$. This is not
the case for the the BCK Lagrangian \cite{MR2476015}, where the canonical momentum for $y_i$
is $\tilde{p}_i=\mathrm{e}^{\alpha (b-t)}m\dot{y}_i$, $i=1,2,3$,
that is different from the kinetic momentum. Now, let us suppose that $L$
is variationally invariant under the transformation
\begin{equation*}
\hat{y}_1=y_1\cos\varepsilon+y_2\sin\varepsilon, \
\hat{y}_2 =-y_1\sin\varepsilon+y_2\cos\varepsilon, \
\hat{y}_3=y_3.
\end{equation*}
Then $\xi_1=y_2$, $\xi_2=-y_1$ and $\xi_3=0$.
For this case Noether's theorem yields
\begin{equation}
\label{eq:fric:CL:mom}
\frac{d}{dt}(m\dot{y}_1y_2-m y_1\dot{y}_2)-\alpha m (\dot{y_1} y_2-y_1\dot{y}_2)=0.
\end{equation}
Note that for $\alpha = 0$ relation \eqref{eq:fric:CL:mom} gives the standard
conservation law $p_1 y_2 - p_2 y_1 = const$ yielded by the classical
Noether's theorem \cite[Section~9.3]{vanBrunt}.


\section*{Acknowledgements}

This work was supported by {\it FEDER} funds through
{\it COMPETE} --- Operational Programme Factors of Competitiveness
(``Programa Operacional Factores de Competitividade'')
and by Portuguese funds through the
{\it Center for Research and Development
in Mathematics and Applications} (University of Aveiro)
and the Portuguese Foundation for Science and Technology
(``FCT --- Funda\c{c}\~{a}o para a Ci\^{e}ncia e a Tecnologia''),
within project PEst-C/MAT/UI4106/2011
with COMPETE number FCOMP-01-0124-FEDER-022690.
Odzijewicz was also supported by FCT through the Ph.D. fellowship
SFRH/BD/33865/2009; Malinowska by Bialystok
University of Technology grant S/WI/02/2011;
and Torres by FCT through the project PTDC/MAT/113470/2009.

The authors are very grateful to two anonymous referees
for their valuable comments and helpful suggestions.


\small



\end{document}